\documentclass{article}
\usepackage[english]{babel}
\usepackage[letterpaper,top=2cm,bottom=2cm,left=2cm,right=2cm,marginparwidth=1.75cm]{geometry}
\usepackage{amsmath}
\usepackage{amsthm}
\newtheorem{definition}{Definition}[section]
\newtheorem{lemma}[definition]{Lemma}
\newtheorem{proposition}[definition]{Proposition}
\newtheorem{theorem}[definition]{Theorem}
\newtheorem{remark}[definition]{Remark}
\newtheorem{corollary}[definition]{Corollary}
\newtheorem{example}[definition]{Example}
\newtheorem{strategy}[definition]{Strategy}

\usepackage{graphicx}
\usepackage[colorlinks=true, allcolors=blue]{hyperref}
\usepackage{float}
\usepackage{amssymb}
\DeclareMathOperator{\tb}{\bf tb_\mathbb{Q}}
\DeclareMathOperator{\rot}{\bf rot_\mathbb{Q}}
\DeclareMathOperator{\tw}{\bf tw}
\DeclareMathOperator{\Fr}{\bf Fr}
\DeclareMathOperator{\cf}{\bf cf}

\newcommand\keywords[1]{\textbf{Keywords}: #1}
\usepackage{authblk}

\title{Legendrian Negative Torus Knots in Universally Tight Lens Spaces}
\author{Han Zhang, \\ Email: \href{1901110026@pku.edu.cn}{1901110026@pku.edu.cn};}
\affil{Department of Mathematics, Peking University, Beijing 100871, P. R. China}

\begin{document}

\maketitle
\begin{abstract}
The main theorem characterizes all Legendrian negative torus knots in universally tight lens space in the sense of coarse equivalence. Together with Onaran's results on Legendrian positive torus knots, all Legendrian torus knots in universally tight lens space is classified. The main method is splitting $L(p,q)$ by convex Heegaard decomposition.  
~\\
~\\
\noindent\keywords{Legendrian torus knots, Universally tight, Lens space}
\end{abstract}

\section{Introduction}
~~~~In a contact 3-manifold, a \textbf{\textit{Legendrian knot}} is a smooth knot which is tangent to contact plane everywhere. There are three classical invariants of a (rationally) null-homologous Legendrian knot: its oriented knot type, (rational) Thurston–Bennequin invariant (denoted by $\tb$), and (rational) rotation number (denoted by $\rot$).

The classification of Legendrian knot for an oriented knot type is very important. People wonder whether all Legendrian knots are determined by there classical invariants in $(S^3, \xi_{st})$. Unfortunately, the answer is NO. Chekanov listed two Legendrian knots in $(S^3, \xi_{st})$ in his paper \cite{C}. These two knots have the same classical invariants and yet they are not Legendrian isotopic. Therefore, it is an interesting problem that: which oriented knot type is good enough such that classical invariants determined Legendrian knots up to Legendrian isotopy in $(S^3, \xi_{st})$? More generally, what about Legendrian links? And what about Legendrian knots in arbitrary contact manifolds?

There has been some progress on this problem. Eliashberg and Fraser shown that Legendrian unknot is determined by its Thurston–Bennequin invariant and rotation number up to Legendrian isotopy in standard tight 3-sphere in \cite{EF}. Later, Etnyre and Honda classified all Legendrian torus knots and figure eight knot in tight 3-sphere in their paper \cite{EH}. Their paper \cite{EH} gives us two ways to consider Legendrian knot classification problem. One is splitting contact manifold by convex Heegaard decomposition; another is studying contact structures on knot complement. The first method is used in this paper since it is natural to split $L(p,q)$ by the Heegaard torus where the torus knot is embedded in.

On the other hand, the study on Legendrian links in $(S^3, \xi_{st})$ are fewer than on knots. The first classification results of Legendrian links in $(S^3, \xi_{st})$ is given by Ding and Geiges in \cite{DG1}. They shown that Legendrian links consisting of an unknot and a cable of that unknot, are classified by their oriented link type and the classical invariants. They also proved the analogous result for torus knots in 1-jet space $J^1(S^1)$ with its standard tight contact structure. In 2021, Dalton, Etnyre, and Traynor gave a classification of Legendrian torus links in \cite{DET}, also, they gave a classification of Legendrian and transversal cable links of knot types that are uniformly thick and Legendrian simple.

There is progress on lens spaces as follow. In \cite{DG2}, Ding and Geiges gave a explicit description of contact mapping class group of $(L(0,1),\xi_{st})$ (i.e.$S^1\times S^2$). Later, Chen, Ding, and Li classified all Legendrian torus knots in $(L(0,1),\xi_{st})$ up to Legendrian isotopy in paper \cite{CDL}. By section 2 of \cite{O}, Heegaard torus in lens space is unique up to smooth isotopy. This fact allow us to define torus knots in arbitrary lens space $L(p,q)$. Onaran gave oriented knot types in theorem 2.3 in \cite{O}. In section 4 of it, Onaran classified all Legendrian positive torus knots in the universally tight contact structures on the lens spaces up to contactomorphism. Recently, Min H. find out the contact mapping class group of universally tight Lens spaces in his paper \cite{M}. His work characterized the contactomorphisms mentioned above.

In this article, we always assume that the lens space is equipped with a universally tight contact structure $\xi_{ut}$ and give the classification of Legendrian negative torus knots in it. Here, ``classification" means the classification of Legendrian knots up to \textbf{\textit{coarse equivalence}}.
\begin{definition}
Two Legendrian knots $K$ and $K'$ are \textbf{coarsely equivalent} if and only if there exists a contactomorphism which is smoothly isotopic to the identity, mapping $K$ to $K'$.
\end{definition}The main idea is to mimic the approach of Etnyre and Honda in their paper \cite{EH}. 
\begin{strategy}\label{strategy}
    Given an oriented knot type $\mathcal{K}$ in lens space:
\begin{itemize}
    \item Find out the max rational Thurston–Bennequin invariant (denoted by $\overline{\tb}$) for knot type $\mathcal{K}$. Then compute all correspondent rational rotation numbers ($\rot$). Classify all Legendrian torus knot of knot type $\mathcal{K}$ with max rational Thurston–Bennequin invariant.
    \item Whenever a Legendrian torus knot $K$ do not reaches its max rational Thurston–Bennequin number, show $K$ destabilizes.
    \item if $K$, $K'$ are two knots realizing the same oriented knot type with max $\tb$. Then study the relationship between their stabilizations.
\end{itemize}
\end{strategy}

Following the steps above, we prove that 
\begin{theorem}[Main theorem]\label{thm1}
In a universally tight contact structure on a lens space $(L(p,q),\xi_{ut})$ with $p>q\geq1$, two Legendrian negative torus knots are coarsely equivalent if and only if their oriented knot types, rational Thurston–Bennequin invariants, and rational rotation numbers agree.
\end{theorem}
Together with Onaran's results (see theorem 4.4 of \cite{O}), all Legendrian torus knots in universally tight lens space is classified. Actually, by corollary 1.4 of \cite{M}, the ``coarsely equivalent" can be replaced by Legendrian isotopy, see \ref{Enhanced main thm}.  As a by-product of the proof of the main theorem, the mountain range of each oriented torus knot type can be easily obtained. 

Another important conception in contact topology is \textbf{\textit{transverse knot}}. A transverse knot is a smooth knot, transverse to the contact plane everywhere. There are two classical invariants of a (rationally) null-homologous transverse knot: its oriented knot type and (rational) self-linking number. As a corollary of \ref{thm1}, we have:
\begin{theorem}
In a universally tight contact structure on a lens space, two transverse negative torus knot are coarsely equivalent if and only if their oriented knot types and rational self-linking numbers are agree.
\end{theorem}

In order to fully understand this paper, the reader should be familiar with section 3 and 4 of \cite{H}, section 4 of \cite{EH} and section 2,3,4 of \cite{O}. As a preparation, section 2 will not including everything used in this paper. After the preparing part, we prove the main theorem in each cases, see section 3,4,5,6. At the last section, we will discuss recent progress based on Min's work \cite{M}.

\section{Preparation}
\subsection{Lens space}
\begin{definition}
Let $p, q, (p>q>0)$ be two co-prime integers. Let $V_1$ and $V_2$ be two solid torus $V_i= D^2\times S^1 $ where $i=1,2$. The meridian and longitude of $V_i$ is denoted by $(\mu_i, \lambda_i)$. In the sense of Heegaard decomposition, a \textbf{lens space} $L(p,q)$ can be described by $V_1\cup_\phi V_2$. The gluing map $\phi:\partial V_2\rightarrow \partial V_1$ ~is an orientation-reversing diffeomorphism given in standard meridian-longitude coordinates on the torus by the matrix
$$
[\phi]=\begin{pmatrix}
-q&q'\\
p&-p'
\end{pmatrix}.$$ where the determinant of $[\phi]$ is $-1$ ,and $p', q'$ are integers. 
\end{definition}

In particular, $\phi(\mu_2)=-q\mu_1+p\lambda_1$. This fact concludes that $$\pi_1(L(p,q))=H_1(L(p,q))=\langle \lambda_1 |~p\lambda_1=1 \rangle.$$
\begin{remark} By section 4.6.1 of \cite{H}, there is a unique $(p', q')$ such that
$$\begin{cases}
p>p'>0,\\
q\geq q'>0,\\
pq'-p'q=1.
\end{cases}$$
In this paper, call such matrix $[\phi]$ the \textbf{standard matrix representation} of gluing map $\phi$. In the following article, we always assume $[\phi]$ has this form .
\end{remark}
\begin{proposition}[prop.4.17 of \cite{H}]
The tight contact structure on $L(p,q)$ is 1-1 determined by the characteristic thickened torus inside it. The \textbf{characteristic thickened torus} is bounded by two convex Heegaard torus $T_1$ and $T_2$ such that 
$$\begin{cases}
\#\Gamma_{T_1}=\#\Gamma_{T_2}=2,\\
\Gamma_{T_2} \text{~has homotopy type } \lambda_2=q'\mu_1-p'\lambda_1,\\
\Gamma_{T_1}  \text{~has homotopy type } \mu_1-\lambda_1,\\
\text{the tight contact structure of this thickened torus is min twist.}
\end{cases}$$
Here, $\#\Gamma_{T_i}$ is the number of components of $\Gamma_{T_i}$. The contact structure of it can be described by the relative Euler class.
\end{proposition}
One can conclude from this proposition that the slope of convex Heegaard torus ranges in $(-\frac{p}{q},0)$.
\begin{proposition}[prop.5.1 of \cite{H}]
If $q\ne p-1$, there are exactly two universally tight contact structures on $L(p,q)$.
If $q=p-1$, there is exactly one.
By assuming the characteristic foliation of $T_1$ and $T_2$ to be standard, the relative Euler class of characteristic thickened torus is $\pm P.D.((-q',p')-(-1,1))$. Here, $P.D.$ means Poincare Duality.
\end{proposition}
\begin{remark}
In this paper, we choose $\xi_{ut}$ to be the one determined by relative Euler class $P.D.((-q',p')-(-1,1))$ of characteristic thickened torus. The characteristic thickened torus has the orientation induced from $L(p,q)$. The upper boundary $T_2$ of the characteristic thickened torus is oriented by $\partial V_1$ while the lower boundary $T_1$ is oriented by the negative orientation of $\partial V_1$.
\end{remark}
\begin{remark}
    For convenient, always assume the convex Heegaard tori is standard foliated (i.e. linear folatied) and described in coordinate $(\mu_1,\lambda_1)$ of $\partial V_1$.
\end{remark}
\subsection{torus knots in lens spaces}

\begin{definition}
An \textbf{oriented knot type} $\mathbf{K_{(a,b)}}$ is represented by a torus knot on a Heegaard torus in $L(p,q)$ with homotopy type $$[K_{(a,b)}]=a\mu_1+b\lambda_1\in H_1(\partial V_1).$$  
\end{definition}

Notice the universal cover $\pi:S^3\rightarrow L(p,q)$ preserves Heegaard decomposition, because $\pi^{-1}(\partial V_1)$ is the boundary of a tubular neighborhood of an unknot in $S^3$. Also, the standard meridian and longitude of this solid torus in $S^3$ is exactly the lift of $\mu_1-$curve and $\mu_2-$curve in $L(p,q)$. Let $r$ denote the order of $K_{(a,b)}$ in the fundamental group of $L(p,q)$. One can compute that:
$$
r[K_{(a,b)}]=(ra+\frac{rbq}{p})\mu_1+\frac{rb}{p}\mu_2\in H_1(\partial V_1).
$$(see lemma 2.5 of \cite{O}). This concludes that $\pi^{-1}(K_{(a,b)})$ is a torus link with $\frac{p}{r}$ components and each component is a $(ra+\frac{rbq}{p}, \frac{rb}{p})$ torus knot. We indicate
\begin{lemma}
Let $K$ be a Legendrian $(a,b)$ torus knot on a Heegaard torus $T$ in universally tight lens space $(L(p,q),\xi_{ut})$. Also, assume $K$ is not a rational unknot. Let $\tw(K, \Fr T)$ denote the twisting number of contact planes along $K$ in the framing given by $T$. We have $\tw(K, \Fr T)\leq0$. From lemma \ref{twnumber}, after a small perturbation fixing $K$, T is convex.
\end{lemma}
\begin{proof}
Because $\pi$ is a local contactomorphism and a $p-$cover map, we have $$\tw(\pi^{-1}K, \Fr(\pi^{-1}T))=p~{\tw(K, \Fr T)}.$$
If $K$ is not a rational unknot, then, $\pi^{-1}K$ is not an unknot. Let $K_i$ denote each component of $\pi^{-1}K$ where $i=1,2,\dots,\frac{p}{r}$. From theorem 4.1 of \cite{EH} and proposition 2.22 of \cite{DET}, we derive $$\tw(\pi^{-1}K, \Fr(\pi^{-1}T))=\sum_{i=1}^{{p}/{r}}{\tw(K_i, \Fr(\pi^{-1}T))}=\frac{p}{r}{\tw(K_i, \Fr(\pi^{-1}T))}\leq0,~\forall i=1,2,\dots,\frac{p}{r}. $$ Therefore, $\tw(K, \Fr T)\leq0.$
\end{proof} 
\begin{remark}
    For a Legendrian rational unknot $K$ in $(L(p,q),\xi_{ut})$, it is not always true that $\tw(K,\Fr T)\leq0$. Specifically, assume $K$ is a Legendrian $(a,1)$ torus knot on Heegaard torus $T$. Then, $\forall k\in\mathbb{Z}$ there exist a Heegaard torus $T'$ containing $K$ s.t. $K$ is a $(a+k,1)$ torus knot on $T'$. We have $\tw(K, \Fr T')=\tw(K, \Fr T)-k$. Hence, a convex Heegaard torus $T$ containing $K$ is always exist.
\end{remark}

\begin{lemma}[prop.3.2 of \cite{O}]\label{twnumber}
Let $L$ be a Legendrian circle on a surface $\Sigma$ and let $\tw(L, \Fr \Sigma)$ denote the twisting
of the contact planes along $L$ measured by the framing given by $\Sigma$. Then $\Sigma$ can be perturbed to be
convex while fixing $L$ if and only if $\tw(L, \Fr \Sigma)\leq0$. If $\Sigma$ is a convex surface with dividing curve $\Gamma$, then
$$\tw(L, \Fr \Sigma)=-\frac{1}{2}\#(L\cap\Gamma)$$
where $\#(L\cap\Gamma)$ is the unsigned count of the intersection number of $L$ and $\Gamma$. Moreover, if $\Sigma$ is a Seifert
surface of a single oriented Legendrian circle $L$, the above formula computes the Thurston–Bennequin invariant
$\tb(L)$ of $L$. In this case, the rotation number $\rot(L)$ of $L$ is
$$\rot(L) = \chi(\Sigma^+)-\chi(\Sigma^-).$$
\end{lemma}

Put $(a,b)$ torus knot $K$ on a convex Heegaard torus $T$ as a Legendrian ruling; and denote the number of components of divides by $\#\Gamma_T=2n, n\in\mathbb{Z}_{>0}$. The slope of T is denoted by $s(T)=-\frac{s}{t}$ where $s,t$ are co-prime positive integers satisfying $-\frac{s}{t}\in(-\frac{p}{q}, 0)$. By lemma \ref{twnumber}, $$\tw(K,\Fr T)=-n\lvert\det\begin{pmatrix}a&t\\b&-s\end{pmatrix}\rvert=-n\lvert as+bt\rvert$$
Together with the fact
$$\tb K=\tw(K,\Fr T)+ab+\frac{b^2q}{p},$$
we derive that a Legendrian $(a,b)$ torus knot $K$ reaches the max $\tb$ that its knot type permits if and only if $\tw(K,\Fr T)$ reaches its max. In order to find out $$\overline{\tw}(K_{(a,b)},\Fr T):=\max\{\tw(K,\Fr T)|K ~\text{is a Legendrian}~ (a,b)~ \text{knot in} ~(L(p,q),\xi_{ut}) \},$$
it suffice to consider when $K$ is a Legendrian ruling curve or a Legendrian divide of $T$. By the paragraph after lemma 4.10 of \cite{EH}, one can always assume that $\#\Gamma_T=2$ when $K$ reaches $\overline{\tw}(K_{(a,b)},\Fr T)$. More precisely, it suffice to find out $$\max\{-|as+bt|:\text{irreducible fraction } -\frac{s}{t}~\text{ranges in}~(-\frac{p}{q},0)\}$$ for fixed number $a$ and $b$.

In order to compute $\rot(K)$, Onaran use the method developed in \cite{EH}. Define $\mathbb{Z}-$linear function $$f_T: H_1(T)\rightarrow\mathbb{Z}$$ as follows: Let $v$ be a nowhere-vanishing section of $\xi_{ut}$, and $w$ be a nowhere-vanishing section of $\xi_{ut}|_T$, such that $w$ is tangent to Legendrian divides and transverse to Legendrian ruling curves. For an oriented close curve $\gamma$ on $T$, let $f_T(\gamma)$ be the rotation number of $w$ relative to $v$ along $\gamma$. Let $K$ be a Legendrian ruling or divides of order $r$ on $T$, then $\rot(K)$ can be computed by the formula below:
\begin{equation}\label{rot}
    r\rot(K)=f_T(r[K])=(ra+\frac{rbq}{p})f_T(\mu_1)+\frac{rb}{p}f_T(\mu_2).
\end{equation}

\begin{remark}Onaran call torus knot $K_{(a,b)}$ \textbf{positive} if $a,b>0$, call it \textbf{negative} if $a\leq0, b\geq0$. In this article, we will always assume that $-a$ and $b$ are relatively prime non-negative integers. In order to prove main theorem \ref{thm1}, we discuss Legendrian negative torus knots in all situations: 
\begin{itemize}
\item $\frac{b}{a}=0$, $K_{(a,b)}$ is unknot $\pm\mu_1$ which is classified in \cite{EF},
\item$\frac{b}{a}\in(-1, 0)$,
\item$\frac{b}{a}\in[-\frac{p'}{q'}, -1]$,
\item$\frac{b}{a}\in(-\frac{p}{q},-\frac{p'}{q'})$,
\item $\frac{b}{a}=-\frac{p}{q}$, $K_{(a,b)}$ is unknot $\pm\mu_2$ which is classified in \cite{EF},
\item $\frac{b}{a}\in(-\infty, -\frac{p}{q})$,
\item
$\frac{b}{a}=-\infty$, $K_{(a,b)}$ is isotopic to $-K_{(1,1)}$ which is classified in \cite{O}.
\end{itemize}
Then we apply the strategy \ref{strategy} in each cases.
\end{remark}

\section{For the case of $\frac{b}{a}\in[-\frac{p'}{q'}, -1]$}
~~~~Firstly, $\overline{\tw}(K_{(a,b)},\Fr T)=0$, since we can choose $K$ to be a Legendrian divide of a convex Heegaard torus $T$ of slope $s(T)=-\frac{s}{t}=\frac{b}{a}$. i.e. We take $s=b$ and $t=-a$. By assumption before, $s,t$ are positive co-prime integers.

Now, assume $K$ is a Legendrian divide of a linear foliated convex Heegaard torus $T$ with $\#\Gamma_T=2$ and slope $s(T)=-\frac{s}{t}=\frac{b}{a}$. To compute $\rot(K)$, we firstly compute $f_T(\mu_i)$ for $i=1,2$. Actually, if $T$ has Legendrian ruling of homotopy type $\mu_i$, then $f_T(\mu_i)$ is exactly the rotation number of a closed Legendrian ruling curve of type $\mu_i$.

Since $\frac{b}{a}\in[-\frac{p'}{q'}, -1]$, one can find a convex Heegaard torus $T_2$ above $T$ with two divides and $s(T_2)=-\frac{p'}{q'}$. The ``above" means that $T$ has the orientation induced from $V_1$ s.t. the ``above" direction of $T$ and the orientation of $T$ forms an orientation the same as $L(p,q)$. Similarly, one can find a convex Heegaard torus $T_1$ below $T$ with two divides and $s(T_1)=-1$. The thickened torus bounded by $T_1$ and $T_2$ is exactly a characteristic thickened torus (denoted by $N$). Let $N_i$ be the thickened torus bounded by $T$ and $T_i$, for $i=1,2$. The relative Euler class of $N, N_1, N_2$ is denoted by $, e_{N_1}, e_{N_2}$. By assumption, $$e_N= P.D.((-q',p')-(-1,1)).$$Hence $N$ is universally tight by prop.5.1 of \cite{H} and so do $N_1$ and $N_2$. Together with the fact that $$e_N= e_{N_1}+ e_{N_2},$$ we have 
$$
\begin{cases}
e_{N_1}=P.D.((-q',p')-(-t,s)),\\
e_{N_2}=P.D.((-t,s)-(-1,1)).
\end{cases}
$$
\subsection{computation of $f_T(\mu_i)$, $i=1,2$}\label{1stcomputation}
~~~~Perturb $T$ and $T_1$ such that they have ruling slope $0$. Let $A_1\cong\gamma\times I$ be a convex annulus with Legendrian boundary and properly embedded in $N_1\cong\partial V_1\times I\cong T^2\times I$. Here, both $\gamma\times0$, $\gamma\times1$ are supposed to be Legendrian ruling curve: $\gamma\times1=T\cap A_1$ and $\gamma\times0=T_1\cap A_1$. Also, assume $A_1$ has the orientation such that $\partial A_1=\gamma\times1-\gamma\times0$ where $\gamma$ is an oriented knot with homotopy type $(1,0)$ on torus. i.e. $\partial A_1$ is the union of a $\mu_1-$curve on $T$ and a $(-\mu_1)-$curve on $T_1$. Then, by prop 4.5 of \cite{H}, we have:
$$
\chi(A_1^+)-\chi(A_1^-)=\langle e_{N_1},(1,0) \rangle=\det\begin{pmatrix}
 -t+1 & 1\\
 s-1 & 0\\
\end{pmatrix}=1-s
$$
where $\langle-,-\rangle$ is Kronecker product. Using lemma \ref{twnumber}, we can also compute that 
$$
\Gamma_{A_1}\cap(\gamma\times1)=-2\tw(\gamma\times1,\Fr{A_1})=-2\tw(\gamma\times1,\Fr{T})=\Gamma_{T}\cap(\gamma\times1)=2|\det\begin{pmatrix}
 1 & -t\\
 0 & s\\
\end{pmatrix}|=2s
$$ and
$$
\Gamma_{A_1}\cap(\gamma\times0)=-2\tw(\gamma\times0,\Fr{A_1})=-2\tw(\gamma\times0,\Fr{T_1})=\Gamma_{T_1}\cap(\gamma\times0)=2|\det\begin{pmatrix}
 1 & -1\\
 0 & 1\\
\end{pmatrix}|=2
.$$
From above computations and tightness of $N_1$, $\Gamma_{A_1}$ must looked like below. $\Gamma_{A_1}$ consists of two arcs connecting $\gamma\times1$ to $\gamma\times0$ and $(s-1)$ boundary parallel dividing curves along $\gamma\times0$. The signature $\pm$ in the figure below is the signature of divergence.  
\begin{figure}[H]
    \centering
    \includegraphics[scale=1]{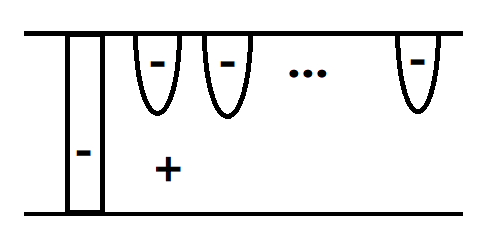}
            \caption{The upper horizontal straight line is $T\cap A_1$; the lower is $T_1\cap A_1$.}
\end{figure}
Let $D_{1,0}$ be the convex meridian disc of solid torus $V_1$ bounded by $T_1$ with orientation s.t. $\partial D_{1,0}=\gamma\times0\subset\partial A_1$. The meiridian and longitude of $V_1$ is exactly $(\mu_1,\lambda_1)$. From the computation of $\Gamma_{A_1}\cap(\gamma\times0)$, we know that $\Gamma_{D_{1,0}}$ is a diameter of $D_{1,0}$. Let $D_1=A_1\cup D_{1,0}$ be the bigger meridian disc bounded by $\mu_1-$curve on $T$, then $D_1$ has orientation compatible with both $A_1$ and $D_{1,0}$. $\Gamma_{D_1}$ consists of exactly $s$ boundary parallel arcs.
\begin{figure}[H]
    \centering
    \includegraphics[scale=1]{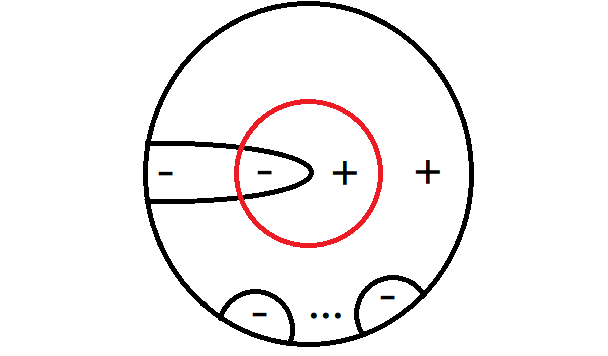}
    \caption{The red is $\gamma\times0$. Inside the red line is $D_{1,0}$ and outside the red line is $A_1$.}
\end{figure}
\noindent Therefore, we have:
$$f_T(\mu_1)=\rot(\gamma\times1, D_1)=\chi(D_1^+)-\chi(D_1^-)=1-s.$$

Similarly, in order to compute $f_T(\mu_2)$, one should perturb $T$ and $T_2$ s.t. they have ruling slope $-\frac{p}{q}$. Let $A_2=\delta\times I$ be a convex annulus with Legendrian boundary properly embedded in $N_2\cong\partial V_1\times I\cong T^2\times I$. Here, both components of $\delta\times\{0,1\}$ are Legendrian ruling curves: $\delta\times1=T_2\cap A_2$ and $\delta\times0=T\cap A_2$. Also, assume $A_2$ has the orientation such that $\partial A_2=\delta\times1-\delta\times0$ and $\delta$ is an oriented knot with homotopy type $(-q,p)$ on torus. i.e. $\partial A_2$ is the union  of a $\mu_2-$curve on $T_2$ and a $(-\mu_2)-$curve on $T$.

We can compute data below:
$$
\Gamma_{A_2}\cap(\delta\times1)=-2\tw(\delta\times1,\Fr{A_2})=-2\tw(\delta\times1,\Fr{T_2})=\Gamma_{T_2}\cap(\delta\times1)=2|\det\begin{pmatrix}
 -q' & -q\\
 p' & p\\
\end{pmatrix}|=2
,$$ 
$$
\Gamma_{A_2}\cap(\delta\times0)=-2\tw(\delta\times0,\Fr{T})=\Gamma_{T}\cap(\delta\times0)=2|\det\begin{pmatrix}
 -q & -t\\
 p & s\\
\end{pmatrix}|=2(pt-qs)>0,
$$and
$$
\chi(A_2^+)-\chi(A_2^-)=\langle e_{N_2},(-q,p) \rangle=\det\begin{pmatrix}
 -q'+t & -q\\
 p'-s & p\\
\end{pmatrix}=-1+pt-qs.$$
From above, $\Gamma_{A_2}$ must consists of two arcs connecting $\delta\times1$ to $\delta\times0$ and $(pt-qs-1)$ boundary parallel arcs along $\delta\times0=T\cap A_2$.

Let $D_{2,0}$ denote the convex meridian disc of solid torus $V_2$ bounded by $T_2$ with orientation s.t. $\partial D_{2,0}=\delta\times1$. The meridian and longitude of $V_2$ is exactly $(\mu_2,\lambda_2)$. From the computation of $\Gamma_{A_2}\cap(\delta\times1)$, $\Gamma_{D_{2,0}}$ is simply a diameter of $D_{2,0}$. Let $D_2=(-A_2)\cup D_{2,0}$, then $\partial D_2=\delta\times0\subset T$ is a $\mu_2-$curve on $T$. By $\Gamma_{D_2}=\Gamma_{A_2}\cup \Gamma_{D_{2,0}}$, we have:
$$f_T(\mu_2)=\rot(\delta\times0, D_2)=\chi(D_2^+)-\chi(D_2^-)=1-(pt-qs).$$
Together above, we conclude:
\begin{theorem}
Assume $-a, b$ are co-prime positive integers and $\frac{b}{a}\in[-\frac{p'}{q'}, -1]$. Let $K$ be a Legendrian $(a,b)$ torus knot with max $\tb$ number in $(L(p,q),\xi_{ut})$. Here, $\xi_{ut}$ is determined by the relative Euler class $P.D.((-q',p')-(-1,1))$ of characteristic thickened torus with standard foliated boundary. Then $$\rot(K)=a+b\frac{1+q}{p}.$$
\end{theorem}
\begin{proof}
this can be reached from the direct computation by formula \ref{rot}.
\end{proof}
\subsection{proof of main theorem in this case}
The main theorem \ref{thm1} in this case can be easily concluded by two following lemmas.
\begin{lemma}
Assume $-a, b$ are co-prime positive integers and $\frac{b}{a}\in[-\frac{p'}{q'}, -1]$. Let $K$ and $K'$ be two Legendrian $(a,b)$ torus knots with max $\tb$ number in $(L(p,q),\xi_{ut})$. Here, $\xi_{ut}$ is determined by the relative Euler class $P.D.((-q',p')-(-1,1))$ of characteristic thickened torus with standard foliated boundary. Then there is a contactomophism $\phi$ s.t. $\phi(K)=K'.$
\end{lemma}
\begin{proof}
Assume $K'$ is a Legendrian divide of a convex Heegaard torus $T'$ with two divides of slope $s(T')=\frac{b}{a}$. Then, there exist a convex Heegaard torus $T_2'$ with two divides of slope $\frac{p'}{q'}$ and a convex Heegaard torus $T_1'$ with two divides of slope $-1$. Let $N_i'$ denote the thickened torus bounded by $T'$ and $T_i'$. Similarly, we define $T$, $N_i$ for $K$. Construct contactomorphism of the relevant layers and glue the maps on each pieces together. The proof is done.

More precisely, if the ruling slope of $T_i$ and $T_i'$ are agree (so do $T$ and $T'$), there is a contactomorphism from $N_i$ to $N_i'$. Also, we can find contactomophism from solid torus $V_i$ to solid torus $V_i'$. Here, $V_i$ (correspondingly, $V_i'$) is a solid torus with its meridian $\mu_i$ and boundry $T_i$ (correspondingly, $T_i'$).
\end{proof}
\begin{lemma}
Assume $-a, b$ are co-prime positive integers and $\frac{b}{a}\in[-\frac{p'}{q'}, -1]$. Let $K'$ be a Legendrian $(a,b)$ torus knot which does not reach max $\tb$ number in $(L(p,q),\xi_{ut})$. Here, $\xi_{ut}$ is determined by the relative Euler class $P.D.((-q',p')-(-1,1))$ of characteristic thickened torus with standard foliated boundary. Then $K'$ destabilizes.
\end{lemma}
\begin{proof}
Let $K'$ sit on a convex Heegaard torus $T'$. Let a convex Heegaard torus $T$ have two divides of slope $\frac{b}{a}$ and let $K$ be a Legendrian divide of $T$. This torus $T$ can be disjoint from $T'$. Let $A$ be a convex annulus with $\partial A=K\cup K'$ properly embedded in the thickened torus bounded by $T$ and $T'$. Notice that $$\tw(K',\Fr A)=\tw(K',\Fr T')<0$$ and $$\tw(K,\Fr A)=\tw(K,\Fr T)=0,$$ thus $\Gamma_A$ cuts off a half disc along $K'$. i.e. $K'$ is a stabilization of some Legendrian $(a,b)$ torus knot.  
\end{proof}
\begin{remark}
    The mountain range in this case has only one top.
\end{remark}
\section{For the case of $\frac{b}{a}\in(-1,0)$}
~~~~Similar with the case of $\frac{b}{a}\in[-\frac{p'}{q'}, -1]$, we have $\overline{\tw}(K_{(a,b)},\Fr T)=0$, since one can choose $K$ to be a Legendrian divide of a convex Heegaard torus $T$ of slope $s(T)=-\frac{s}{t}=\frac{b}{a}$. We still assume $s=b$ and $t=-a$ are positive integers where $s,t$ are relatively prime. There exist a unique positive integer $m$ s.t.$-\frac{1}{m+1}\geq-\frac{s}{t}>-\frac{1}{m}$.

In order to compute $f_T$, we mimic the approach in section 4.2 of \cite{EH}. Take $T_1$ and $T_2$ above $T$ which described in the former section. Let $T_{1,m+1}$ (relatively, $T_{1,m}$) be a convex Heegaard torus having two divides of slope $-\frac{1}{m+1}$ (relatively, $-\frac{1}{m}$). $T$ is between $T_{1,m+1}$ and $T_{1,m}$. Still, Let $N_1$ denote the thickened torus bounded by $T$ and $T_1$, although $N_1$ is outside characteristic thickened torus $N$ in this case. Let $N_{1,m+1}$ be the thickened torus bounded by $T$ and $T_{1,m+1}$ and let $V_1$ be the solid torus bounded by $T_1$ with its meridian and longitude $(\mu_1,\lambda_1)$.

Secondly, let us find out the relative Euler class of these layers.
\begin{proposition}[prop.4.13 of \cite{H}]\label{prop.4.13}
Assume $\xi$ is a min-twist tight contact structure on $M=T^2\times I$ with boundary condition:
$$\begin{cases}
s_1=\text{the slope of } T^2\times 1=-1,\\
s_0=\text{the slope of } T^2\times 0=-\frac{1}{m}, m\in\mathbb{Z}_{\geq2},\\
\#\Gamma_{T^2\times i}=2, i=0,1,\\
{T^2\times i}=2 \text{are linear foliated for } i=0,1.
\end{cases}$$
Then, $\xi$ is $1-1$ determined by the relative Euler class $e_M$, or equivalently, 
$$\langle e_M, A\rangle=\chi(A^+)-\chi(A^-)=\#\text{positive half-disc}-\#\text{negative half-disc}$$
where $A=0\times S^1\times I$ is a convex annulus with Legendrian boundary and $\#\text{positive/negative half-disc}$ is the number of positive/negative half-discs which cut off by $\Gamma_A$.
\end{proposition}
\begin{proof}
Do an diffeomorphism on $T^2\times I$ given by $(x,y,t)\rightarrow(y,x,1-t)$. Then the boundary conditions of $\xi$ is the same as in prop.4.12 of \cite{H}. The dividing set $\Gamma_A$ is looked like below. $\Gamma_A$ consists of two dividing curves connecting $0\times S^1\times1$ to $0\times S^1\times0$ and $m-1$ boundary parallel dividing curves along $0\times S^1\times0$.
\begin{figure}[H]
    \centering
    \includegraphics[scale=1]{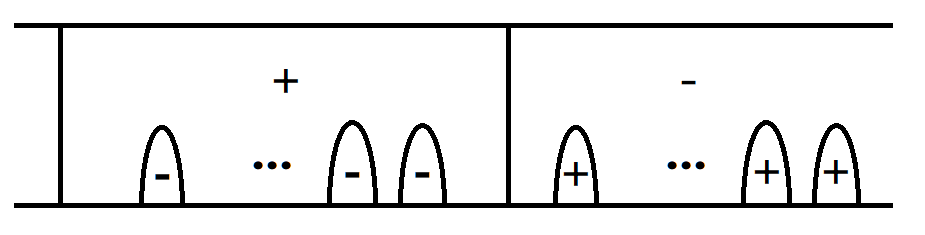}
            \caption{the upper horizontal straight line is $0\times S^1\times1$; the lower is $0\times S^1\times0$.}
\end{figure}
\noindent We have $$\#\text{positive half-disc}+\#\text{negative half-disc}=m-1.$$

Therefore, $e_M=P.D.(x,0)$ , where $x$ ranges in m values: $m-1,m-3,\dots,-(m-1)$. We can also write $e_M$ in another form (see prop.4.22 of \cite{H})
$$e_M=P.D.\{[(1,-1)-(k,-1)]+[-(k,-1)+(m,-1)]\}$$
or
$$e_M=P.D.\{[-(1,-1)+(k,-1)]+[(k,-1)-(m,-1)]\}$$
where $k$ ranges in $1,2,\dots,m$. From lemma 4.14 of \cite{H}, $(T^2\times I,\xi)$ is a contact structure obtained from gluing two universally tight layers together which relative Euler class are $P.D.[(1,-1)-(k,-1)]$ and $P.D.[-(k,-1)+(m,-1)]$. Or equivalently, it is obtained by attaching positive(negative) bypasses along slope-$\infty$-ruling on $T^2\times0$ first, then the negatives(positives). 
\end{proof}
\begin{remark}
By analyzing the standard tubular neighborhood $V_1$ of a Legendrian circle and its stabilizations, all contact structures described above can be  embedded in $V_1$. Back into the situation in this paper, let $T_{1,k}$ be the convex Heegaard torus with two divides of slope $-\frac{1}{k}$. Here, $k$ is described in lemma \ref{prop.4.13}. Without loss of generality, assume that $T$ is in the universally tight layer described in \ref{prop.4.13} which is bounded by $T_{1,k}$ and $T_{1,m+1}$. The universally tight layer has the relative Euler class $P.D.[-(k,-1)+(m+1,-1)]$ or $P.D.[(k,-1)-(m+1,-1)]$ where $(k=1,2,\dots,m)$.
\end{remark}
\subsection{computation of $f_T(\mu_i)$, $i=1,2$}
~~~~If $T$ is in the universally tight layer as above with relative Euler class $P.D.[-(k,-1)+(m+1,-1)]$, for convenient, say $T$ is in a \textit{\textbf{positive universally tight layer}}, we have $$P.D.[-(k,-1)+(m+1,-1)]=P.D.\{[-(k,-1)+(t,-s)]+[-(t,-s)+(m+1,-1)]\}$$
which concludes that
$$e_{N_{1,m+1}}=P.D.[-(t,-s)+(m+1,-1)]$$ since a sub-layer of a universally tight thickened torus is also universally tight.
By a similar approach in \ref{1stcomputation}, we can compute that $$f_T(\mu_1)=1-s.$$

In order to compute $f_T(\mu_2)$, we first compute the relative Euler class of ${N\cup N_1}$:
$$P.D.e_{N\cup N_1}=[(-q',p')-(-1,1)]+[(1,-1)-(k,-1)]+[-(k,-1)+(t,-s)]=(-q'+t-2k+2,p'-s).$$ The first term is for $N$, the second term is for the layer bounded by $T_1$ and $T_{1,k}$ ,and the third term is for the layer bounded by $T_{1,k}$ and $T$.
Then, take a convex annulus $A_2$ with Legendrian boundary properly embedded in to $N\cup N_1$, i.e.
$$(N\cup N_1, A_2)\cong(T^2\times I, \gamma\times I)$$
where $\gamma$ is a curve with slope the same as $\mu_2$. We have
$$
\chi(A_2^+)-\chi(A_2^-)=\langle e_{N\cup N_1},(-q,p) \rangle=\det\begin{pmatrix}
 -q'+t-2k+2 & -q\\
 p'-s & p\\
\end{pmatrix}=p(t-2k+2)-qs-1.
$$
After completing $A_2$ in to a meridian disc bounded by $\gamma\times0$ , one can conclude that $$f_T(\mu_2)=-p(t-2k+2)+qs+1$$
where $k=1,2,\dots,m$.

If $T$ is in the universally tight layer with relative Euler class $P.D.[(k,-1)-(m+1,-1)]$, for convenient, say $T$ is in a \textit{\textbf{negative universally tight layer}}, we have results below:
$$f_T(\mu_1)=s-1$$
and the relative Euler class of $N\cup N_1$ is given by
$$P.D.e_{N\cup N_1}=[(-q',p')-(-1,1)]+[-(1,-1)+(k,-1)]+[+(k,-1)-(t,-s)]=(-q'-t+2k,p'+s-2)$$
and thus
$$f_T(\mu_2)=p(t-2k)-q(s-2)+1.$$
Taking these data into formula \ref{rot}, we have
\begin{theorem}
Assume $-a, b$ are co-prime positive integers and $\frac{b}{a}\in(-1,0)$. Let m be the unique positive integer s.t.$-\frac{1}{m+1}\geq-\frac{s}{t}=\frac{b}{a}>-\frac{1}{m}$. Let $K$ be a Legendrian $(a,b)$ torus knot with max $\tb$ number in $(L(p,q),\xi_{ut})$. Here, $\xi_{ut}$ is determined by the relative Euler class $P.D.((-q',p')-(-1,1))$ of characteristic thickened torus with standard foliated boundary. Then $$\rot(K)=a+b\frac{1+q}{p}+2b(k-1),~k=1,2,\dots,m.$$ Or
$$\rot(K)=-a+b\frac{1+q}{p}-2bk,~k=1,2,\dots,m.$$
By assuming $b\ne1$, or equivalently,  $K$ is not a rational unknot, these $2m$ values are distinct to each and thus is $1-1$ correspondent to the different cases of layering.
\end{theorem}
\begin{remark}
The first formula of rotation number represents the case where $T$ lies in the positive universally tight layer and the second formula of rotation number is for the case where $T$ lies in the negative universally tight layer.
\end{remark}

\subsection{proof of main theorem in this case}
~~~~We can conclude the main theorem \ref{thm1} in this case by following lemmas. In this subsection, we mimic the approach in section 4.2 of \cite{EH}.
\begin{lemma}
Assume $-a, b$ are co-prime positive integers and $\frac{b}{a}\in(-1,0)$. Also, assume $b\ne1$. Let $K$ and $K'$ be two Legendrian $(a,b)$ torus knots with max $\tb$ number and same rotation number $\rot$ in $(L(p,q),\xi_{ut})$. Here, $\xi_{ut}$ is determined by the relative Euler class $P.D.((-q',p')-(-1,1))$ of characteristic thickened torus with standard foliated boundary. Then, there is a contactomophism $\phi$ s.t. $\phi(K)=K'$.
\end{lemma}
\begin{proof}
Assume $K$($K'$) lies on convex Heegaard torus $T$($T'$) with two divides as a Legendrian divide. $T_{1,m+1}$ and $T_{1,m+1}'$ are defined as above. Then, we can construct contactomorphism layer by layer and glue them together since the contact structure on these layers are 1-1 determined by the rotation number of $K$($K'$).
\end{proof}
\begin{lemma}
Assume $-a, b$ are co-prime positive integers and $\frac{b}{a}\in(-1,0)$. Also, assume $b\ne1$. Let $m$ and $e$ be the positive integers such that $|a|=bm+e$ where $0<e<b$. Let $K$ and $K'$ be Legendrian $(a,b)$ torus knots with max $\tb$. If the rotation numbers are $r$ and $r-2e$, then $S_e^+(K')$ and $S_e^-(K)$ are contactomorphism-equivalent to each. If the rotation numbers are $r$ and $r-2(b-e)$, then $S_{(b-e)}^+(K')$ and $S_{(b-e)}^-(K)$ are contactomorphism-equivalent to each.
\end{lemma}
\begin{proof}
Let $\rot(K')=\rot(K)-2e$. By the discussion above, one can assume $K'$ lies on a linear foliated convex Heegaard torus $T'$ with two divides of slope $\frac{b}{a}$. Let $T_1'$ be a convex Heegaard torus with two divides of slope $-1$ above $T'$; let $T_{1,m}'$ be a linear foliated convex Heegaard torus with two divides of slope $-\frac{1}{m}$ between $T'$ and $T_1'$. The layer bounded by $T_{1,m}'$ and $T'$ has the relative Euler class $P.D.[-(-a,-b)+(m,-1)]$. The relative Euler class of the layer bounded by $T_1'$ and $T_{1,m}'$ is $P.D.[(1+m-2k,0)]$ where $k=1,2,\dots,m$. After a small perturbation, one can assume $T_{1,m}'$ has ruling slope $\frac{b}{a}$. Let $A$ be a convex annulus properly embedded into this layer with $\partial A=\gamma - K$ where $\gamma$ is a Legendrian $(a,b)$ torus knot lying on $T_{1,m}'$ as a ruling curve. We have $$\chi(A^+)-\chi(A^-)=\langle P.D.[-(-a,-b)+(m+1,-1)],(a,b)\rangle=e$$ and $$\tw(\gamma,\Fr{A})=|\det\begin{pmatrix}
 m & a\\
 -1 & b\\
\end{pmatrix}|=e.$$
Thus, $\Gamma_A$ consists of $e$ boundary parallel arcs along $\gamma$. Therefore, $\gamma=S_e^+(K')$.

For knot $K$, define $T_1$, $T_{1,m}$ and $T$ similarly. By the same approach, the reader can find out that the Legendrian $(a,b)$ torus knot $\delta$ lying on $T_{1,m}$ as a ruling curve is exactly $\delta=S_e^-(K)$. And the relative Euler class of the layer bounded by $T_1$ and $T_{1,m}$ is also $P.D.[(1+m-2k,0)]$ which is the same as the relative Euler class of the layer bounded by $T_1'$ and $T_{1,m}'$. Thus one can construct a contactomorphism layer by layer, mapping $\gamma$ to $\delta$. i.e. There is a contactomorphism mapping $S_e^+(K')$ to $S_e^-(K)$.

The second case is left for reader, and it is not difficult once you notice the relation of $K$ and the Legendrian $(a,b)$ knot lying on $T_{1,m+1}$ as a ruling curve.
\end{proof}
\begin{lemma}
Assume $-a, b$ are co-prime positive integers and $\frac{b}{a}\in(-1,0)$. Let $K'$ be a Legendrian $(a,b)$ torus knot which does not reach max $\tb$ number in $(L(p,q),\xi_{ut})$. Here, $\xi_{ut}$ is determined by the relative Euler class $P.D.((-q',p')-(-1,1))$ of characteristic thickened torus with standard foliated boundary. Then $K'$ destabilizes.
\end{lemma}
\begin{proof}
the proof is totally the same with the case where $\frac{b}{a}\in[-\frac{p'}{q'}, -1]$.
\end{proof}

The reader can conclude \ref{thm1} in this case from above lemmas by totally the  same approach of theorem 4.13 in \cite{EH}. This is because $(a,b)$ torus knot type has the mountain range with the same shape as $(a,b)$ torus knot in $(S^3,\xi_{st})$. Therefore, the proof of main theorem \ref{thm1} makes no difference with theorem 4.13 of \cite{EH}.
\section{For the case of $\frac{b}{a}\in(-\frac{p}{q},-\frac{p'}{q'})$}
~~~~Notice that $\lambda_1-\mu_1$ is the most outward contact longitude of $V_1$ and $\lambda_2$ is the most outward contact longitude of $V_2$. If we exchange the position of $V_1$ and $V_2$ and assume they are equipped with new coordinate $(\mu_i', \lambda_i')$ where $i=1,2$, then $(\mu_i', \lambda_i')$ should satisfy 
$$
\begin{cases}
\mu_1=\mu_1',\\
\mu_2=\mu_2',\\
\lambda_1-\mu_1=\lambda_1',\\
\lambda_2=\lambda_2'-\mu_2'.
\end{cases}
$$
Together with
$$
\begin{cases}
\mu_1=p'\mu_2+p\lambda_2,\\
\lambda_1=q'\mu_2+q\lambda_2,
\end{cases}
$$
We derive that under new coordinate $(\mu_i', \lambda_i')$, the gluing map $\phi^{-1}:\partial V_1\rightarrow \partial V_2$ has standard matrix representation:
$$[\phi^{-1}]=\begin{pmatrix}
-(p-p')&(p-p')-(q-q')\\
p&-(p-q)
\end{pmatrix}.
$$
Especially, $\phi^{-1}(\mu_1')=-(p-p')\mu_2'+p\lambda_2'$. Therefore, $L(p,q)=L(p,p-p')$, and the universally tight contact structure on $L(p,p-p')$ induced from $(L(p,q),\xi_{ut})$ is determined by the relative Euler class of its characteristic thickened torus:$$P.D.\{(-[(p-p')-(q-q')],p-q)-(-1,1)\}.$$

We can also find out the transition matrix from coordinate $(\mu_1,\lambda_1)$ to $(\mu_2',\lambda_2')$, given by
$$
\begin{pmatrix}
-(p-p')&-(q-q')\\
p&q
\end{pmatrix}.$$
This concludes that an $(a,b)$ torus knot in $L(p,q)$ is actually an $(\tilde{a},\tilde{b})$ torus knot in $L(p,p-p')$ where $$(\tilde{a},\tilde{b})=(-(p-p')a-(q-q')b,~pa+qb).$$ 
Once $\frac{b}{a}\in(-\frac{p}{q},-\frac{p'}{q'})$, we have $\frac{\tilde{b}}{\tilde{a}}\in(-1,0)$. Applying the same argument on $(\tilde{a},\tilde{b})$ torus knot in $L(p,p-p')$, the main theorem \ref{thm1} is also holds when $\frac{b}{a}\in(-\frac{p}{q},-\frac{p'}{q'})$. The reader can check the correctness of this section as an easy exercise of linear algebra.
\section{For the case of $\frac{b}{a}\in(-\infty, -\frac{p}{q})$}
~~~~The first task is to compute $\overline{\tw}(K_{(a,b)},\Fr T)$. As before, we assume $T$ is a convex Heegaard torus with two divides of slope $-\frac{s}{t}$, and $K$ is a knot lying on $T$ as a Legendrian ruling curve of slope $\frac{b}{a}$. Here, $-a$, $b$ are positive co-prime integers and so do $s$ and $t$. 
Thus,
$${\tw}(K_{(a,b)},\Fr T)=-|\det\begin{pmatrix}
t&a\\
-s&b
\end{pmatrix}|=-\det\begin{pmatrix}
t&a\\
-s&b
\end{pmatrix}=-(as+b{t})$$,
and it suffice to find the min of $(as+bt)$ where $-\frac{s}{t}\in(-\frac{p}{q},0)$ and both $a$ and $b$ are fixed.

We take a quick review on Onaran's result:
\begin{lemma}[see theorem 4.2 of \cite{O}]
Assume $a$ and $b$ are relatively prime non-negative integers and so do $s$ and $t$. For fixed $a$ and $b$, let $-\frac{s}{t}$ range in $(-\frac{p}{q},0)$. If $a$ and $b$ are positive, then $|\det\begin{pmatrix}
t&a\\
-s&b
\end{pmatrix}|$ reaches its min if and only if $s=t=1$.
If $(a,b)=(1,0)$ or $(0,1)$, $\min|\det\begin{pmatrix}
t&a\\
-s&b
\end{pmatrix}|=1$.

\end{lemma}
As a simple application of it, we have the corollary below:
\begin{corollary}
Let $-\frac{u_1}{v_1}$, $-\frac{u_2}{v_2}$ be two rational numbers such that vector pair $(-v_1,u_1),(-v_2,u_2)$ forms a positive integral basis for $\mathbb{Z}^2$. i.e. $u_1v_2-u_2v_1=1 $. Here, $u_i$ and $v_i$ are non-negative integers for $i=1,2$. Assume $a$, $b$ are fixed co-prime integers. If $\frac{b}{a}\in(-\frac{u_1}{v_1},-\frac{u_2}{v_2})$. Then, when irreducible fraction $-\frac{s}{t}\in (\mathbb{Q}\cup\{\infty\})\backslash(-\frac{u_1}{v_1},-\frac{u_2}{v_2})$, $|\det\begin{pmatrix}
t&a\\
-s&b
\end{pmatrix}|$ reaches its min if and only if $-\frac{s}{t}=-\frac{u_2-u_1}{v_2-v_1}$. If $\frac{b}{a}=-\frac{u_i}{v_i}$, $|\det\begin{pmatrix}
t&a\\
-s&b
\end{pmatrix}|$ has minimal value $1$ and $-\frac{s}{t}=-\frac{u_2-u_1}{v_2-v_1}$ is a possible value whenever $|\det\begin{pmatrix}
t&a\\
-s&b
\end{pmatrix}|$ reaches its min. Here, it is allowed that these fractions to be $\infty$. 
\begin{proof}
It suffice to consider an orientation preserving linear map taking $(-v_1,u_1)\mapsto(1,0)$ and $(-v_2,u_2)\mapsto(0,1)$. Then use the lemma above.
\end{proof}
\end{corollary}

\begin{remark}
These fractions in the corollary above can be $\infty$. Especially, if $-\frac{u_1}{v_1}=\infty$, then $-\frac{u_1}{v_1}$ should be regard as $-\infty=-\frac{1}{0}$. 
\end{remark}
For convenient, we define several conceptions below. The reader should take care that they are not official names.
\begin{definition}
Let $s_0$ and $s_1$ be two rational slopes (including $\infty$) marked on the boundary of Tessellation disc. If the related direction of these two slopes forms an integral basis, then there is an \textbf{connecting arc} connecting $s_0$ and $s_1$ in Tessellation disc. 
\end{definition}
These connecting arcs are properly embedded into this disc and are actually geodesic curves of this hyperbolic disc with Poincare metric. More, all these connecting arcs are disjoint in the interior of this disc. In figure \ref{Tess disc}, the connecting arcs is the arcs inside of the disc. 
\begin{figure}[H]
    \centering
    \includegraphics[scale=1]{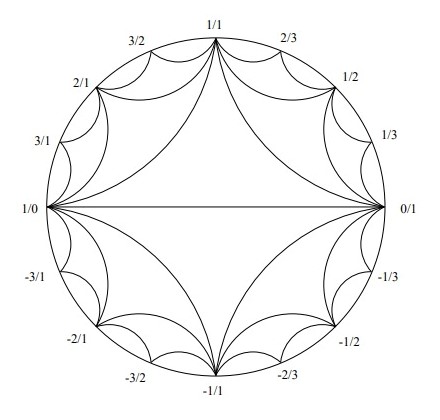}
    \caption{Tessellation disk. This picture is stolen from \cite{H}
}\label{Tess disc}
\end{figure}

In order to apply the corollary above, we divide interval $(-\infty, -\frac{p}{q})$ in to finite segments. 

\begin{definition}
A \textbf{hop sequence} is a sequence of slopes $s_0,s_1,\dots,s_n$ going counterclockwise on the boundary of Tessellation disc such that there is a connecting arc connecting $s_{i-1}$ to $s_i$ for each $i=1,2,\dots,n$. As in lemma 4.12 of \cite{H}, such sequence is called a \textit{ \textbf{hop sequence}} from $s_0$ to $s_n$.
\end{definition}Let $-\infty=-\frac{p_0}{q_0}<-\frac{p_1}{q_1}<\dots<-\frac{p_n}{q_n}=-\frac{p}{q}$ be a hop sequence such that for each $i=0,1,\dots,n-1$, integral vectors $(-q_i,p_i)$ and $(-q_{i+1},p_{i+1})$ form an integral basis for $\mathbb{Z}^2$. Also, for each $i=0,1,\dots,n$, assume $q_i$ and $p_i$ are co-prime non-negative integers.

Recall the conclusion in section 4.4.4 of \cite{H}, the continued fraction of $-\frac{p}{q}$ induce the shortest hop sequence from $-\frac{p}{q}$ to $-1$. In order to find a hop sequence from $-\infty$ to $-\frac{p}{q}$, it suffice to find an isomorphism $\Phi$ of Tessellation disc, mapping boundary arc $[-\infty, -1]$ to $[-\infty, -1]$ and reverse the orientation, Since we can induce this hop sequence from the continued fraction of $\Phi(-\frac{p}{q})$. Fortunately, such $\Phi$ exists.

Consider linear map $\mathcal{A}: x\in\mathbb{R}^{2\times1}\mapsto Ax$ where $A=\begin{pmatrix}
1&1\\
0&-1
\end{pmatrix}$. By regarding the boundary of Tessellation disc as $\mathbb{R}P^1$, $\mathcal{A}$ gives an isomorphism of Tessellation disc, denoted by $\tilde{\mathcal{A}}$. The reader can check that
$$\begin{cases}
\tilde{\mathcal{A}}(-1)=\infty,\\ \tilde{\mathcal{A}}(\infty)=-1,\\
\tilde{\mathcal{A}}([-\infty, -1])=[-\infty, -1],\\
\tilde{\mathcal{A}}(-\frac{p}{q})=-\frac{p}{p-q}.
\end{cases}$$
Especially, the first identity means that $\mathcal{A}$ maps a straight line of slope $-1$ to a line with slope $\infty$. The reader can also check that the inverse of $\mathcal{A}$ is exactly itself. Thus, the inverse of $\tilde{\mathcal{A}}$ is itself too.

Assume $-\frac{p}{p-q}$ has the following continued fraction expansion:
\begin{equation}
-\frac{p}{p-q}=
 r_1 - \cfrac{1}{r_2
          - \cfrac{1}{r_3
          -\dots  \cfrac{1}{r_m} } } 
=:\cf(r_1, r_2, \dots, r_m)
\end{equation} with all $r_i\in \mathbb{Z}_{\leq-2}$. Then, the shortest hop sequence from $-\frac{p}{p-q}$ to $-1$ can be obtained by decreasing the last entry of the continued fraction. After mapping this sequence by $\tilde{\mathcal{A}}$, we get the shortest hop sequence from $-\infty$ to $-\frac{p}{q}$.
\begin{example}
Let $(p,q)=(8,5)$, then $\tilde{\mathcal{A}}(-\frac{8}{5})=-\frac{8}{8-5}=-\frac{8}{3}=-3-\frac{1}{-3}$. The shortest hop sequence from $-\frac{8}{3}$ to $-1$ is $-3-\frac{1}{-3}$, $-3-\frac{1}{-2}$, $-3-\frac{1}{-1}$, $-1$, i.e. $-\frac{8}{3}<-\frac{5}{2}<-2<-1$. Mapping this sequence by $\tilde{\mathcal{A}}$, we get $-\frac{8}{8-3}>-\frac{5}{5-2}>-\frac{2}{2-1}>-\frac{1}{1-1}$. We have $-\infty<-2<-\frac{5}{3}<-\frac{8}{5}$ is the shortest hop sequence from $-\infty$ to $-\frac{8}{5}$.
\end{example}

Let $-\infty=-\frac{p_0}{q_0}<-\frac{p_1}{q_1}<\dots<-\frac{p_n}{q_n}=-\frac{p}{q}$ be the shortest hop sequence from $-\infty$ to $-\frac{p}{q}$ described as above. It remains to check that $-\frac{p_i-p_{i-1}}{q_i-q_{i-1}}\in(-\frac{p}{q},0)$ holds for each $i=1,2,\dots,n$. Notice that the continued fraction expansion has a matrix representation, its corresponding integral vector is described as below:
$$
\tilde{\mathcal{A}}(-\frac{p}{q})=-\frac{p}{p-q}=\cf(r_1,r_2,\dots,r_m) \sim \begin{pmatrix}
0&1\\
-1&r_1
\end{pmatrix}
\begin{pmatrix}
0&1\\
-1&r_2
\end{pmatrix}
\dots
\begin{pmatrix}
0&1\\
-1&r_{m-1}
\end{pmatrix}
\begin{pmatrix}
1\\
r_m
\end{pmatrix}.$$
This equation means that this integral vector has slope $\cf(r_1,r_2,\dots,r_m)$. Thus, the corresponding integral vector of $$\tilde{\mathcal{A}}(-\frac{p_i-p_{i-1}}{q_i-q_{i-1}})=-\frac{p_i-p_{i-1}}{(p_i-p_{i-1})-(q_i-q_{i-1})}=-\frac{p_i-p_{i-1}}{(p_i-q_i)-(p_{i-1}-q_{i-1})}$$ is of form:
$$
\begin{pmatrix}
0\\
1
\end{pmatrix} \text{or} \begin{pmatrix}
0&1\\
-1&r_1
\end{pmatrix}
\begin{pmatrix}
0&1\\
-1&r_2
\end{pmatrix}
\dots
\begin{pmatrix}
0&1\\
-1&r_k
\end{pmatrix}
\begin{pmatrix}
0\\
1
\end{pmatrix}, k=1,2,\dots,m-1.
$$
Therefore, $\tilde{\mathcal{A}}(-\frac{p_i-p_{i-1}}{q_i-q_{i-1}})$ ranges in
$$ -\infty, \cf(r_1), \cf(r_1,r_2), \dots, \cf(r_1,r_2,\dots,r_{m-1})
.$$
The readers should notice that $\tilde{\mathcal{A}}(-\frac{p_i-p_{i-1}}{q_i-q_{i-1}})$ doesn't take each value in above if $r_i=-2$ for some $i$.  
Also, by induction of $m$, it is easy to check that
$$ -\infty<\cf(r_1)<\cf(r_1,r_2)<\dots< \cf(r_1,r_2,\dots,r_{m-1})<\cf(r_1,r_2,\dots,r_m)
.$$
Hence, 
$-\frac{p_i-p_{i-1}}{q_i-q_{i-1}}$ ranges in
$$ -1=\tilde{\mathcal{A}}(-\infty), \tilde{\mathcal{A}}(\cf(r_1)), \tilde{\mathcal{A}}(\cf(r_1,r_2)), \dots, \tilde{\mathcal{A}}(\cf(r_1,r_2,\dots,r_{m-1}))
$$
and there holds:
$$ -1=\tilde{\mathcal{A}}(-\infty)> \tilde{\mathcal{A}}(\cf(r_1))> \tilde{\mathcal{A}}(\cf(r_1,r_2))> \dots> \tilde{\mathcal{A}}(\cf(r_1,r_2,\dots,r_{m-1}))>
\tilde{\mathcal{A}}(\cf(r_1,r_2,\dots,r_m))=-\frac{p}{q}
$$
i.e. we have proved that $-\frac{p_i-p_{i-1}}{q_i-q_{i-1}}\in(-\frac{p}{q},-1]$ holds for each $i=1,2,\dots,n$.

More precisely, it holds that $-\frac{p_i-p_{i-1}}{q_i-q_{i-1}}\in[-\frac{p'}{q'},-1]$ for each $i=1,2,\dots,n$. For convenient, slopes $-\frac{p_i-p_{i-1}}{q_i-q_{i-1}}$ ($i=1,2,\dots,n$) is called the \textit{\textbf{delta slopes}} for hop sequence $-\infty=-\frac{p_0}{q_0}<-\frac{p_1}{q_1}<\dots<-\frac{p_n}{q_n}=-\frac{p}{q}$. The reader can check that $0, -\infty=-\frac{p_0}{q_0},-\frac{p_1}{q_1},\dots,-\frac{p_n}{q_n}=-\frac{p}{q}$ is the shortest hop sequence from $0$ to $-\frac{p}{q}$. This means that every step of this sequence is going counterclockwise and the furthest in Tessellation disc. Also, $-1$ together with $-\frac{p_i-p_{i-1}}{q_i-q_{i-1}}, i=1,2,\dots, n$ is the delta slopes of this hop sequence. By considering these directions of delta slopes in new coordinate $(\mu_2',\lambda_2')$, these slopes is still in $(-\frac{p}{p-p'},-1]$ since we can do the same analyse in $L(p,p-p')$ and the induced map of $GL_2(\mathbb{Z})$ preserves the delta slopes of a hop sequence. Therefore, by mapping these directions back to coordinate $(\mu_1,\lambda_1)$, we have $-\frac{p_i-p_{i-1}}{q_i-q_{i-1}}\in[-\frac{p'}{q'},-1]$, for $i=1,2,\dots,n$.

From above discussions, we have:
\begin{theorem}
Let $-\infty=-\frac{p_0}{q_0}<-\frac{p_1}{q_1}<\dots<-\frac{p_n}{q_n}=-\frac{p}{q}$ be the shortest hop sequence such that for each $i=0,1,\dots,n$, $q_i$ and $p_i$ are co-prime non-negative integers. Let $\frac{b}{a}\in(-\infty, -\frac{p}{q})$ where $-a$ and $b$ are fixed co-prime non-negative integers. If $\frac{b}{a}\in(-\frac{p_{i-1}}{q_{i-1}}, -\frac{p_i}{q_i})$ for some $i=1,2,\dots,n$, then $$\overline{\tw}(K_{(a,b)},\Fr T)=-|\det\begin{pmatrix}
a&-(q_i-q_{i-1})\\
b&p_i-p_{i-1}
\end{pmatrix}|$$ and the only possible value for $s(T)$ is $-\frac{p_i-p_{i-1}}{q_i-q_{i-1}}$ whenever ${\tw}(K_{(a,b)},\Fr T)$ reaches its max.
If $\frac{b}{a}=-\frac{p_i}{q_i}$ for some $i=1,2,\dots,n-1$, then $\overline{\tw}(K_{(a,b)},\Fr T)=-1$, both $-\frac{p_i-p_{i-1}}{q_i-q_{i-1}}$ and $-\frac{p_{i+1}-p_i}{q_{i+1}-q_i}$ are possible values for $s(T)$ whenever ${\tw}(K_{(a,b)},\Fr T)$ reaches its max.
\end{theorem}
\begin{remark}
If $\frac{b}{a}=-\frac{p_i}{q_i}$ for some $i=1,2,\dots,n-1$, and $\tw(K_{(a,b)},\Fr T)$ reaches its maximum $\overline{\tw}(K_{(a,b)},\Fr T)=-1$, both $-\frac{p_i-p_{i-1}}{q_i-q_{i-1}}$ and $-\frac{p_{i+1}-p_i}{q_{i+1}-q_i}$ are possible values for $s(T)$, but maybe not all possible values for $s(T)$. See the proof of lemma \ref{relation of those with max tw} below.
\end{remark}
As before, the main theorem in this case can be easily concluded from the following two lemmas.
\begin{lemma}
Assume $-a, b$ are co-prime positive integers and $\frac{b}{a}\in(-\infty, -\frac{p}{q})$. Let $K'$ be a Legendrian $(a,b)$ torus knot which does not reach max $\tb$ number in $(L(p,q),\xi_{ut})$. Here, $\xi_{ut}$ is determined by the relative Euler class $P.D.((-q',p')-(-1,1))$ of characteristic thickened torus with standard foliated boundary. Then $K'$ is a stabilization of some Legendrian $(a,b)$ torus knot, thus it can be regard as several stabilizations of one with max $\tb$ number.
\begin{proof}
the proof is totally similar with the cases above. The only difference is the twisting number of $K$ and $K$ relative to framing of $T$.
\end{proof}
\end{lemma}
\begin{lemma}\label{relation of those with max tw}
Assume $-a, b$ are co-prime positive integers and $\frac{b}{a}\in(-\infty, -\frac{p}{q})$. Let $K$ and $K'$ be two Legendrian $(a,b)$ torus knots with max $\tb$ number in $(L(p,q),\xi_{ut})$. Here, $\xi_{ut}$ is determined by the relative Euler class $P.D.((-q',p')-(-1,1))$ of characteristic thickened torus with standard foliated boundary. Then, there is a contactomophism $\phi$, s.t. $\phi(K)=K'$. As a corollary, the mountain range of Legendrian ${(a,b)}$ torus knot has only one top. The rotation number of the Legendrian $(a,b)$ torus knots with max $\tb$ number is determined by formula \ref{rot}, the corresponding data can be found in section \ref{1stcomputation}.
\begin{proof}
Let $-\infty=-\frac{p_0}{q_0}<-\frac{p_1}{q_1}<\dots<-\frac{p_n}{q_n}=-\frac{p}{q}$ be the shortest hop sequence. For each $i=0,1,\dots,n$, let $q_i$ and $p_i$ be co-prime non-negative integers.

If $\frac{b}{a}\in(-\frac{p_{i-1}}{q_{i-1}}, -\frac{p_i}{q_i})$ for some $i=1,2,\dots,n$, then, after a small perturbation, we can assume $K$($K'$) is a Legendrian ruling curve on a convex Heegaard torus $T$($T'$) with two divides of slope $s(T)=-\frac{p_i-p_{i-1}}{q_i-q_{i-1}}$. Noticing that $s(T)=s(T')\in[-\frac{p'}{q'},-1]$, thus $T$ is in a characteristic thickened torus. Thus, by constructing contactomorphism layer by layer as before, we can reach the conclusion.

The difficult case is $\frac{b}{a}=-\frac{p_i}{q_i}$ for some $i=1,2,\dots,n-1$. Consider continued fraction of $-\frac{p}{p-q}=\tilde{\mathcal{A}}(-\frac{p}{q})=\cf(r_1,r_2,\dots,r_m)$ and the hop sequence from $-\frac{p}{p-q}$ to $-1$ induced by adding $+1$ on the last position of continued fraction. If $r_k\ne-2$, then $\cf(r_1,\dots,r_{k-1},r_k+1)$, $\cf(r_1,\dots,r_{k-1},r_k+2)$ are appeared in this hop sequence with its delta slope $\cf(r_1,\dots,r_{k-1})$. Each two of the corresponding integral vector of these three slopes forms an integral basis for $\mathbb{Z}^2$. In another word, in Tessellation disc, there is a connecting arc connecting each two of these three slopes. If $r_k\ne-2$, $r_{k+1}=r_{k+2}=\dots=r_l=-2$ and $r_{l+1}\ne-2$ where $l\geq k+1$, then $\cf(r_1,\dots,r_l,-2)$, $\cf(r_1,\dots,r_l,-1)=\cf(r_1,\dots,r_k,-2,\dots,-2,-1)=\dots=\cf(r_1,\dots,r_k,-1)=\cf(r_1,\dots,r_k+1)$, $\cf(r_1,\dots,r_k+2)$ are appeared in the hop sequence from $-\frac{p}{p-q}$ to $-1$ with its delta slopes $\cf(r_1,\dots,r_l)$ and $\cf(r_1,\dots,r_{k-1})$. By writing these slopes in form of corresponding integral vector, the reader can check the following facts in Tessellation disc:
\begin{itemize}
    \item There is a connecting arc connecting $\cf(r_1,\dots,r_l,-2)$ to $\cf(r_1,\dots,r_l)$.
    \item There is a connecting arc connecting $\cf(r_1,\dots,r_k+2)$ to $\cf(r_1,\dots,r_{k-1})$.
    \item There is a connecting arc connecting $\cf(r_1,\dots,r_l,-1)=\cf(r_1,\dots,r_k,-2,\dots,-2,-1)=\dots=\cf(r_1,\dots,r_k,-1)=\cf(r_1,\dots,r_k+1)$ to each of $\cf(r_1,\dots,r_{k-1})$, $\cf(r_1,\dots,r_k)$, $\dots$, $\cf(r_1,\dots,r_l)$.
    \item $\cf(r_1,\dots,r_{k-1})$, $\cf(r_1,\dots,r_k)$, $\dots$, $\cf(r_1,\dots,r_l)$ is a hop sequence.
\end{itemize}
In other cases, Tessellation disc at $-\frac{p_i}{q_i}$ is also looked like below. The only difference is the number of these point $D_i$, but there is at least one $D_i$. The check of its correctness is left for reader.
Mapping slopes above by $\tilde{\mathcal{A}}$, we have the figure below:
\begin{figure}[H]
    \centering
    \includegraphics[scale=0.5]{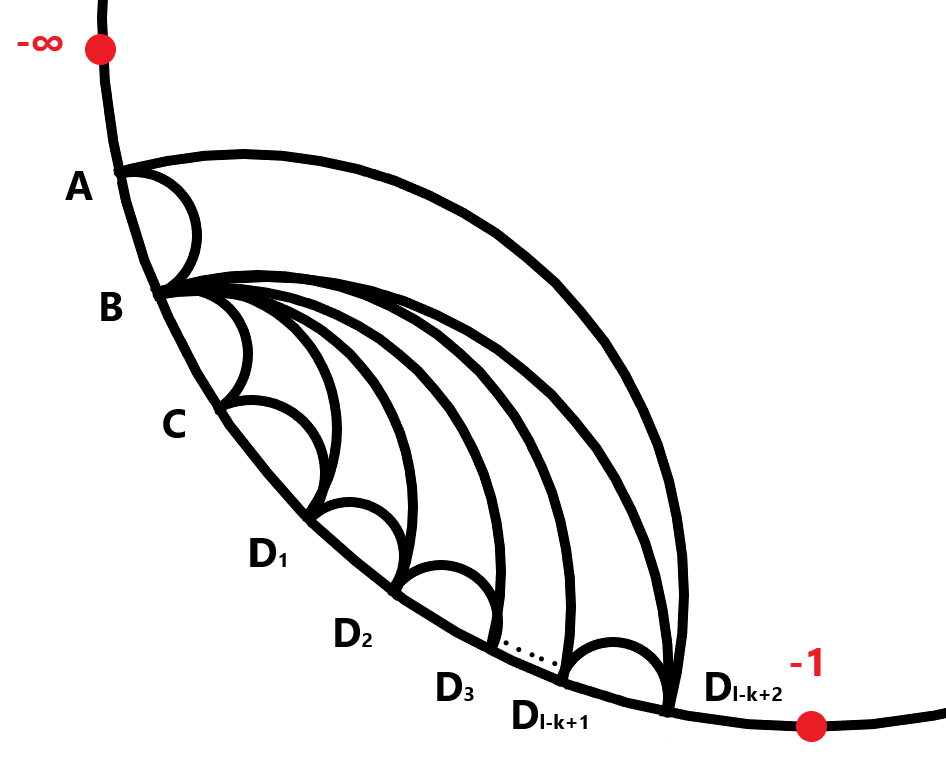}
     \caption{the slopes in Tessellation disk.}
\end{figure}
where:
\begin{itemize}
    \item point $A=-\frac{p_{i-1}}{q_{i-1}}=\tilde{\mathcal{A}}(\cf(r_1,\dots,r_k+2))$,
    \item point $B=-\frac{p_i}{q_i}=\tilde{\mathcal{A}}(\cf(r_1,\dots,r_k+1))$,
    \item point $C=-\frac{p_{i+1}}{q_{i+1}}=\tilde{\mathcal{A}}(\cf(r_1,\dots,r_l,-2))$,
    \item point $D_1=-\frac{p_{i+1}-p_i}{q_{i+1}-q_i}$,
    \item point $D_{l-k+2}=-\frac{p_i-p_{i-1}}{q_i-q_{i-1}}$,
    \item points $D_{l-k+2},D_{l-k+1},\dots,D_1$ are actually $\tilde{\mathcal{A}}(\cf(r_1,\dots,r_{k-1})), \tilde{\mathcal{A}}(\cf(r_1,\dots,r_k)), \dots, \tilde{\mathcal{A}}(\cf(r_1,\dots,r_l))$.
\end{itemize}

First of all, whenever $K$ reaches its max twist number (equivalently, $\tb$ number) $\overline{\tw}(K_{(a,b)},\Fr T)=-1$, the values of $s(T)$ are exactly these $D_i$ where $i=1,2,\dots,l-k+2$. Arguing by contradiction, assume there is another convex Heegaard torus $T$ of slope $s(T)$ such that $K$ is lying on it with ${\tw}(K,\Fr T)=-1$. Then, represent $s(T)\in(-\frac{p}{q},0)$ on Tessellation disc by point $D$ and there is a connecting arc $\stackrel\frown{BD}$ from point $B$ to point $D$. Observing Tessellation disc, this arc $\stackrel\frown{BD}$ is impossible to disjoint from the union of connecting arcs $\stackrel\frown{CD}_1$, $\stackrel\frown{D_iD}_{i+1}$ (where $i=1,2,\dots,l-k+1$), $\stackrel\frown{AD}_{l-k+2}$ in the interior of Tessellation disc. This reaches a contradiction since the connecting arcs of slopes are disjoint in the interior of Tessellation disc.

Now assume $K$ is a Legendrian ruling curve on convex Heegaard torus $T$ of two divides with slope $s(T)=D_i$. Correspondingly, assume $K'$ is a Legendrian ruling curve on convex Heegaard torus $T'$ disjoint from $T$ of two divides with slope $s(T')=D_{i+1}$ ($i=1,2,\dots,l-k+1$). Noticing that there is a connecting arc from $D_i$ to $D_{i+1}$. Thus the thickened torus $M$ bounded by $T$ and $T'$ is actually a basic slice. Let $\Sigma$ be a convex annulus properly embedded into this slice with its boundary $K\cup K'$. We have $$\Gamma_\Sigma\cap K=-2\tw(K,\Fr{\Sigma})=-2\tw(K,\Fr{T})=2$$ and $$\Gamma_\Sigma\cap K'=-2\tw(K',\Fr{\Sigma})=-2\tw(K',\Fr{T'})=2.$$ Thus, the dividing set $\Gamma_\Sigma$ of $\Sigma$ is exactly two arcs connecting $K$ to $K'$. See the figure below. After a perturbation fixing boundary, the characteristic foliation of $\Sigma$ can be made linear as in the figure below, and the direction of characteristic foliation is given by the arrow.
\begin{figure}[H]
    \centering
    \includegraphics[scale=1]{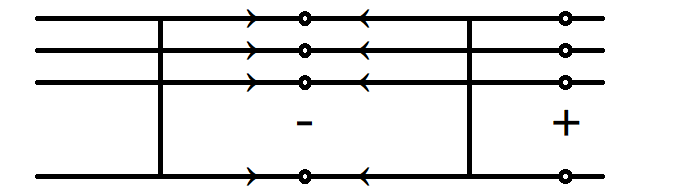}
    \caption{The top horizontal boundary line is $K$ and the bottom is $K'$.}
\end{figure}
The Legendrian isotopy from $K$ to $K'$ is given by the characteristic foliation. Thus, we can always assume $s(T)=D_1\in[-\frac{p'}{q'},-1]$ for $K$, and the left work is still construct contactomorphism layer by layer for two Legendrian $(a,b)$ knots with max $\tb$ number.
\end{proof}
\end{lemma}

\begin{example} Let $(p,q)=(8,5)$ and $(a,b)=(-3,7)$. From the last example, we know that the shortest hop sequence from $-\infty$ to $-\frac{p}{q}=-\frac{8}{5}$ is $-\infty<-2<-\frac{5}{3}<-\frac{8}{5}$. Its delta slope sequence is $-1>-\frac{3}{2}$. Notice that $\frac{b}{a}=-\frac{7}{3}\in(-\infty,-2)$, thus $\overline{\tw}(K_{(-3,7)},\Fr T)=-|\det\begin{pmatrix}
-3&-1\\
7&1
\end{pmatrix}|=-4$. Taking this value back to formula of $\tb$ number, we have $\overline{\tb}(K_{(-3,7)})=-4+(-3)7+\frac{7^2 5}{8}=\frac{45}{8}$. Assume $K$ is a Legendrian $(-3,7)$ knot with max $\tb$ number. Then, its rotation number is given by formula \ref{rot} where $T$ is the convex Heegaard torus that $K$ lies on. We have $s(T)=-1$, taking the data in section \ref{1stcomputation}, we have $\rot(K)=-\frac{7}{4}$.
\begin{remark}
    The mountain range in this case has only one top. This fact is compatible with Onaran's results in section 4 of \cite{O}.
\end{remark}

\end{example}
\section{Recent progress}
Thanks to Min's work, the contact mapping class group of $(L(p,q), \xi_{ut})$ is characterized by the theorem below. It gives the answer for question 1 of \cite{O}.
\begin{theorem}[theorem 1.1 of \cite{M}]
Let $\pi_0(Cont(L(p,q),\xi_{ut}))$ be the the group of contact isotopy classes of co-orientation preserving contactomorphisms , then 
$$\pi_0(Cont(L(p,q),\xi_{ut}))=\begin{cases}
\mathbb{Z}_2, p\ne-2 \text{ and } q\equiv-1 \text{(mod p)}  \\
\mathbb{Z}_2, q\not\equiv\pm 1 \text{(mod p) and } q^2\equiv 1\text{(mod p) }\\
1, \text{otherwise}.
\end{cases}$$
For the first two cases, the contact mapping class group is generated by the involution $\overline{\sigma}$. Here, the involution $\overline{\sigma}$ is induced from $\sigma:S^3\rightarrow S^3\subset\mathbb{C}^2$ mapping $(z_1,z_2)$ to $(z_2,z_1)$.
\end{theorem}
Min pointed out that:
\begin{corollary}[corollary 1.4 of \cite{M}] For every universally tight Lens space $(L(p,q),\xi_{ut})$, including $(L(0,1),\xi_{st})$, once a co-orientation preserving contactomorphism $f$ is smoothly isotopic to the identity, $f$ is contact isotopic to the identity.
    
\end{corollary}
As a corollary of corollary 1.4 of \cite{M}, We have an enhanced version of Main theorem:
\begin{theorem}[Enhanced Main Theorem]\label{Enhanced main thm} 
For every universally tight Lens space $(L(p,q),\xi_{ut})$, including $(L(0,1),\xi_{st})$, two Legendrian torus knots are Legendrian isotopic if and only if their classical invariants are agree.
\end{theorem}
In the theorem above, the Legendrian torus knots in $(L(0,1),\xi_{st})$ is classified in \cite{CDL}. Unlike $(L(p,q),\xi_{ut})$, classical invariants for Legendrian torus knots in $(L(0,1),\xi_{st})$ are oriented knot type, twisting number, and rotation number. Here, the twisting number is $\tw(K,\Fr T)$ for Legendrian torus knot $K$ and Heegaard torus $T$; the rotation number is the rotation number of the tangent vector of $K$ measured by a trivialization of $\xi_{st}$.   
\section{Acknowledgment}
\begin{itemize}
    \item Thank my advisor Fan Ding. Under his guidance, I finished reading Honda's paper\cite{H} which really plays a very important role in my work. Also, without his suggestion, I would not even have a try on researching on torus knots in lens space. I sincerely wish him in good health. 
    \item Thank John. B. Etnyre and Hyunki Min for their precious advices. They read throughout of my handwritten papers. I am really appreciated that they would like to spare their time to have a look on my work. This really matters a lot for a young reseacher.
    \item Thank XiaoLong Hu for his long term financial assistant.
\end{itemize}

\end{document}